\documentclass[11pt,british]{amsart}

\usepackage{babel,eucal,url,amssymb,
enumerate,amscd,
}

\theoremstyle{plain}
\newtheorem{thm}{Theorem}[section]

\newtheorem{pro}[thm]{Proposition}

\theoremstyle{definition}
\newtheorem{defn}[thm]{Definition}

\theoremstyle{remark}

\newcommand{\Gtwo}{\ifmmode{{\rm G}_2}\else{${\rm G}_2$}\fi}

\date{\today}

\begin{document}

\title[]
 {Almost para-Hermitian and almost paracontact metric structures induced by natural Riemann extensions}

\author[C. L. Bejan]{Cornelia-Livia Bejan}

\address{
''Gh. Asachi'' Technical University of Iasi \\ Department of Mathematics \\ 700506 Iasi \\ Romania}
\email{bejanliv@yahoo.com}

\author[G. Nakova]{Galia Nakova}

\address{
University of Veliko Tarnovo "St. Cyril and St. Methodius" \\ Faculty of Mathematics and Informatics\\   Department of Algebra and Geometry
\\ 2 Teodosii Tarnovski Str. \\ Veliko Tarnovo 5003\\ Bulgaria}
\email{gnakova@gmail.com}




\begin{abstract}
In this paper we consider a manifold $(M,\nabla )$ with a symmetric linear connection $\nabla $ which induces on the cotangent bundle $T^*M$ of $M$ a semi-Riemannian metric $\overline g$ with a neutral signature. The metric $\overline g$ is called  natural Riemann extension and it is a generalization (made by M. Sekizawa and O. Kowalski) of the Riemann extension, introduced by E. K. Patterson and A. G. Walker (1952). We construct two almost para-Hermitian structures on $(T^*M,\overline g)$ which are almost para-K\"ahler or para-K\"ahler and prove that the defined almost para-complex structures are harmonic. On certain hypersurfaces of $T^*M$ we construct almost paracontact metric structures, induced by the obtained almost para-Hermitian structures. We determine the classes of the corresponding almost  paracontact metric manifolds according to the classification given by S. Zamkovoy and G. Nakova (2018). We obtain a necessary and sufficient condition the considered manifolds to be paracontact metric, K-paracontact metric or para-Sasakian.
\end{abstract}

\newcommand{\g}{\mathfrak{g}}
\newcommand{\s}{\mathfrak{S}}
\newcommand{\F}{\mathcal{F}}
\newcommand{\R}{\mathbb{R}}
\newcommand{\U}{\mathbb{U}}
\newcommand{\G}{\mathbb{G}}
\newcommand{\diag}{\mathrm{diag}}
\newcommand{\End}{\mathrm{End}}
\newcommand{\im}{\mathrm{Im}}
\newcommand{\id}{\mathrm{id}}
\newcommand{\Hom}{\mathrm{Hom}}

\newcommand{\Rad}{\mathrm{Rad}}
\newcommand{\rank}{\mathrm{rank}}
\newcommand{\const}{\mathrm{const}}
\newcommand{\tr}{{\rm tr}}
\newcommand{\ltr}{\mathrm{ltr}}
\newcommand{\codim}{\mathrm{codim}}
\newcommand{\Ker}{\mathrm{Ker}}

\newcommand{\thmref}[1]{Theorem~\ref{#1}}
\newcommand{\propref}[1]{Proposition~\ref{#1}}
\newcommand{\corref}[1]{Corollary~\ref{#1}}
\newcommand{\secref}[1]{\S\ref{#1}}
\newcommand{\lemref}[1]{Lemma~\ref{#1}}
\newcommand{\dfnref}[1]{Definition~\ref{#1}}


\newcommand{\ee}{\end{equation}}
\newcommand{\be}[1]{\begin{equation}\label{#1}}

\maketitle

\section{Introduction}\label{sec-1}
The geometry of an almost para-Hermitian manifold $(N,P,g)$ is determined by the action of the almost 
para-complex structure $P$ as an anti-isometry with respect to the semi-Riemannian metric $g$ in each tangent fibre. The metric $g$ is necessarily of neutral signature. A classification of the almost para-Hermitian manifolds is made by C.-L. Bejan in \cite{B}. The geometry of the almost paracontact metric manifolds is a natural extension of the geometry of the almost para-Hermitian manifolds to the odd dimensional case. Twelve basic classes of almost paracontact metric manifolds $(M,\varphi ,\xi ,\eta ,g)$ with respect to the covariant derivative of the structure tensor $\varphi $ is obtained by S. Zamkovoy and G. Nakova in \cite{ZN}. Moreover, in \cite{ZN} it is shown that 3-dimensional almost paracontact metric 
manifolds belong only to four basic classes from the classification and examples for each of these classes are constructed.
\par
Let $(M,\nabla )$ be an $n$-dimensional manifold endowed with a symmetric linear connection $\nabla $. Patterson and Walker defined in \cite{PW} a semi-Riemannian metric  on the cotangent bundle $T^*M$ of 
$(M,\nabla )$, called  Riemann extension. This metric is of neutral signature $(n,n)$ and it was generalized by M. Sekizawa and O. Kowalski in \cite{KS, S} to natural Riemann extension $\overline g$ which has the same signature. Recently, the metric $\overline g$ has been studied from different points of view. For instance, Bejan and  Kowalski characterized in \cite{BK} some harmonic functions on $(T^*M,\overline g)$. In 
\cite{B1} Bejan and Eken defined a canonical almost para-complex structure on $(T^*M,\overline g)$ and investigated its harmonicity with respect to $\overline g$. In \cite{B2} the authors constructed a family of hypersurfaces of $(T^*M,\overline g)$ which are Einstein manifolds with a positive scalar curvature.
\par
Our aim in the present work is to obtain new examples of almost para-Hermitian and almost paracontact metric  manifolds. The paper is organized as follows. In Sect. 2 we recall some notions and results about the cotangent bundle of a manifold and the lifting of objects from the base manifold to its cotangent bundle. In Sect. 3, motivated from the fact that the natural Riemann extension  $\overline g$ on $T^*M$ is of neutral signature, we construct two almost para-Hermitian structures $(P,\overline g)$ and $(P_1,\overline g_1)$ on $T^*M$, where $\overline g$ and $\overline g_1$ are proper and non-proper natural Riemann extension, respectively. We prove that in the case when $M$ is not flat (resp. $M$ is flat) both manifolds 
$(T^*M,P,\overline g)$ and $(T^*M,P_1,\overline g_1)$ are almost para-K\"ahler (resp. para-K\"ahler). Moreover, we establish that the defined almost para-complex structures $P$ and $P_1$ are harmonic with respect to $\overline g$ and $\overline g_1$, respectively. In Sect. 4 we study a family of non-degenerate hypersurfaces $\widetilde H_t$ of $(T^*M,P,\overline g)$. They are a generalization of the family $H_t$ of non-degenerate hypersurfaces of $(T^*M,\overline g)$, introduced in \cite{B2}. On a hypersurface 
$\widetilde H_t$  with a time-like unit normal vector field we define an almost paracontact metric structure
$(\varphi ,\overline \xi ,\eta ,g)$ induced from the almost para-Hermitian structure $(P,\overline g)$. We determine the classes to which belong the obtained almost paracontact metric manifolds $(\widetilde H_t,\varphi ,\overline \xi ,\eta ,g)$ and give a necessary and sufficient condition the considered manifolds to be paracontact metric. Also, we consider the almost paracontact metric manifolds $H_t$ and obtain a necessary and sufficient condition they to be  para-Sasakian or K-paracontact metric.

\section{Preliminaries}\label{sec-2}
Let $M$ be a connected smooth $n$-dimensional manifold ($n\geq 2$). The cotangent bundle $T^*M$ of $M$ consists of all pairs $(x,\omega )$, where $x\in M$ and $\omega \in T^*_xM$. Let $p : T^*M\longrightarrow M$, \, $p(x,\omega )=x$, \,  be the natural projection of $T^*M$ to $M$. Any local chart 
$(U;x^1,\ldots ,x^n)$ on $M$ induces a local chart $(p^{-1}(U);x^1,\ldots ,x^n,x^{1*},\ldots ,x^{n*})$ on $T^*M$, where for any $i=1,\ldots ,n$ the function $x^i\circ p$ on $p^{-1}(U)$ is identified with the function $x^i$ on $U$ and $x^{i*}=\omega _i=\omega \left(\left(\frac{\partial }{\partial x^i}\right)_x\right)$ at any point $(x,\omega )\in p^{-1}(U)$. The vectors $\{(\partial _1)_{(x,\omega )},\ldots ,
(\partial _n)_{(x,\omega )},(\partial _{1*})_{(x,\omega )},\ldots ,(\partial _{n*})_{(x,\omega )}\}$, where we put $\partial _i=\frac{\partial }{\partial x^i}$ and $\partial _{i*}=\frac{\partial }{\partial \omega ^i}$ \, $(i=1,\ldots ,n)$ \, form a basis of the tangent space $(T^*M)_{(x,\omega )}$ at each point $(x,\omega )\in T^*M$. The Liouville type vector field $W$ is globally defined vector field on  $T^*M$  which is expressed in local coordinates by
\[
W=\sum_{i=1}^{n}\omega _i\partial _{i*} .
\]
Everywhere here we will denote by $\mathcal{F}(M)$, $\chi (M)$ and $\Omega ^1(M)$ the set of all smooth real functions, vector fields and differential 1-forms on $M$, respectively.
\par Now, we recall the constructions of the vertical and complete lifts for which we refer to \cite{YI, YP}.
\par
The vertical lift $f^V$ on $T^*M$ of a function $f\in \mathcal{F}(M)$ is a function on $T^*M$ defined by 
$f^V=f\circ p$. The vertical lift $X^V$ on $T^*M$ of a vector field $X\in \chi (M)$ is a function  on 
$T^*M$ (called evaluation function) defined by 
\[X^V(x,\omega )=\omega (X_x) \, \, \text {or equivalently} \, \, X^V(x,\omega )=\omega _iX^i(x), \, \,
\text {where} \, \, X=X^i\partial _i .
\] 
In \cite{YP} it is shown that a vector field $U\in \chi (T^*M)$ is determined by its action on all evaluation functions. More precisely, the following proposition is valid:
\begin{pro}\label{Proposition 2.1}\cite{YP}
Let $U_1$ and $U_2$ be vector fields on $T^*M$. If $U_1(Z^V)=U_2(Z^V)$ holds for all $Z\in 
\chi (M)$, then  $U_1=U_2$.
\end{pro}
The vertical lift $\alpha ^V$ on $T^*M$ of a differential 1-form $\alpha \in \Omega ^1(M)$ is a tangent vector field to $T^*M$ which is defined by
\[
\alpha ^V(Z^V)=\left(\alpha (Z)\right)^V, \quad Z\in \chi (M) .
\]
In local coordinates we have
\[
\alpha ^V=\sum_{i=1}^{n}\alpha _i\partial _{i*} ,
\]
where $\alpha =\sum_{i=1}^{n}\alpha _i{\rm d}x^i$. Hence, identifying $f^V\in \mathcal{F}(T^*M)$ with $f\in \mathcal{F}(M)$, we obtain $\alpha ^V(f^V)=0$ for all $f\in \mathcal{F}(M)$.
\par
The complete lift $X^C$ on $T^*M$ of a vector field $X\in \chi (M)$ is a tangent vector field to $T^*M$ which is defined by
\[
X^C(Z^V)=[X,Z]^V, \quad Z\in \chi (M) .
\]
In local coordinates $X^C$ is written as
\[
X^C_{(x,\omega )}=\sum_{i=1}^{n}X^i(x)(\partial _i)_{(x,\omega )}-\sum_{h,i=1}^{n}\omega _h
(\partial _iX^h)(x)(\partial _{i*})_{(x,\omega )} ,
\]
where $ X=X^i\partial _i$. Therefore we have $X^C(f^V)=(Xf)^V$ for all $f\in \mathcal{F}(M)$.
\par
We note that the tangent space $T_{(x,\omega )}T^*M$ of $T^*M$ at any point $(x,\omega )\in T^*M$
is generated by the vector fields of the form $\alpha ^V+X^C$.

\section{Almost para-Hermitian structures induced by natural Riemann extensions}\label{sec-3}
This section deals with para-Hermitian geometry and first  we will recall some basic notions.
\par
An $(1,1)$ tensor field $P$ on a $2n$-dimensional smooth manifold $N$ is said to be {\it an almost product structure} if $P\neq \pm {\rm Id}$ and $P^2={\rm Id}$. In this case the pair $(N,P)$ is called {\it an almost product manifold}. An almost product structure $P$ on $N$ such that the eigendistributions of $P$ corresponding to the eigenvalues  $\pm 1$ of $P$ have the same rank, is called {\it a para-complex structure} and $(N,P)$ - {\it an almost para-complex manifold}. 
\par
A $2n$-dimensional smooth manifold 
$N$ has {\it an almost para-Hermitian structure} $(P,g)$ if it is endowed with an almost para-complex structure $P$ and a semi-Riemannian metric $g$ such that $P$ is an anti-isometry with respect to $g$, i.e. 
$g(PX,PY)=-g(X,Y)$, \, $X, Y \in \chi (N)$. The manifold $(N,P,g)$ is called {\it an almost para-Hermitian manifold}. The metric $g$ is necessarily indefinite of a neutral signature. The fundamental 2-form $\Omega$ and the tensor field $F$ of type $(0,3)$ of an almost para-Hermitian manifold are defined by $\Omega (X,Y)=g(X,PY)$ and  
$F(X,Y,Z)=g((\nabla _XP)Y,Z)$, respectively, where $\nabla $ is the Levi-Civita connection of $g$. The tensor field $F$ has the following properties:
\begin{equation}\label{3.0}
F(X,Y,Z)=-F(X,Z,Y) , \quad F(X,PY,PZ)=F(X,Y,Z) , \, \, X, Y, Z \in \chi (N).
\end{equation} 
A classification of the almost para-Hermitian manifolds  is given in \cite{B}. Here we recall the characteristic conditions of two basic classes of almost para-Hermitian manifolds:
\begin{itemize}
\item $(N,P,g)$ is para-K\"ahler if $\nabla P=0 \, \Longleftrightarrow F=0$;
\item $(N,P,g)$ is almost para-K\"ahler if ${\rm d}\Omega (X,Y,Z)=0 \Longleftrightarrow \\
\mathop{\s} \limits_{(X,Y,Z)}F(X,Y,Z)=0$,
where $\mathop{\s} \limits_{(X,Y,Z)}$ denotes the cyclic sum over $X, Y, Z$.
\end{itemize}
In this section we also need the following notion introduced in \cite{G-RV}:
\begin{defn}\label{Definition 3.1}
Any (1,1)-tensor field $T$ on a (semi-) Riemannian manifold $(N,h)$ is called harmonic if $T$ viewed as an endomorphism field $T : (TN,h^C)\longrightarrow (TN,h^C)$ is a harmonic map, where $h^C$ denotes the complete lift (see \cite{YI}) of the (semi-) Riemannian metric $h$.  
\end{defn}
We recall 
\begin{pro}\label{Proposition 3.1}\cite{G-RV}
Let $(N,h)$ be a (semi-) Riemannian manifold and let $\nabla $ be the Levi-Civita connection of $h$. Then any (1,1)-tensor field $T$ on $(N,h)$ is harmonic if and only if $\delta T=0$, where
\[
\delta T={\rm trace}_h(\nabla T)={\rm trace}_h\{(X,Y)\longrightarrow (\nabla _XT)Y\} .
\]
\end{pro}

\par
Further, if it is not otherwise stated, we assume that $(M,\nabla )$ is an n-dimensional manifold endowed with a symmetric linear connection $\nabla $ (i. e. $\nabla $ is torsion-free). In \cite{S} Sekizawa constructed a semi-Riemannian metric $\overline g$ at each point $(x,\omega )$ of the cotangent bundle $T^*M$ of $M$ by:
\begin{equation}\label{3.1}
\begin{array}{ll}
\overline g_{(x,\omega)}(X^C,Y^C)=-a\omega (\nabla _{X_x}Y+\nabla _{Y_x}X)+b\omega (X_x)\omega (Y_x) , \\
\overline g_{(x,\omega)}(X^C,\alpha ^V)=a\alpha _x(X_x) , \\
\overline g_{(x,\omega)}(\alpha ^V,\beta ^V)=0 
\end{array}
\end{equation} 
for all vector fields $X, Y$  and all differential 1-forms $\alpha , \beta$ on $M$, where $a, b$ are arbitrary constants. We may assume $a>0$ without loss of generality. The semi-Riemannian metric $\overline g$ defined by \eqref{3.1} is called {\it a natural Riemann extension} \cite{KS, S}. When $b\neq 0$ $\overline g$ is called {\it a proper natural Riemann extension}. In the case when $a=1$ and $b=0$ we obtain the notion of the {\it classical Riemann extension} defined by Patterson and Walker (see \cite{PW, W}). In \cite{B2} it is shown that $\overline g$ is of neutral signature $(n,n)$. 
\par
In \cite{B1} authors have constructed a canonical almost para-complex structure $\mathcal{P}$ on $T^*M$ by $\mathcal{P}X^C=X^C$ and $\mathcal{P}\alpha ^V=-\alpha ^V$, where $X^C$ and $\alpha ^V$ are the complete lift of a vector field $X$ and the vertical lift of a differential 1-form $\alpha $ on $M$, respectively. They proved that $\mathcal{P}$ is harmonic if and only if the natural Riemann extension $\overline g$ on $T^*M$ is non-proper. 
\par
In this section we shall construct almost para-complex structures $P$ and $P_1$ on $T^*M$ such that 
$(P,\overline g)$ and $(P_1,\overline g_1)$ are almost para-Hermitian structures on $T^*M$, where 
$\overline g$ (resp. $\overline g_1$) is the proper (resp. non-proper) natural Riemann extension on 
$T^*M$. Moreover, we show that $P$ and $P_1$ are harmonic with respect to $\overline g$ and 
$\overline g_1$, respectively. 
\par
The following conventions and formulas will be used later on.
\par
Let $T$ be an $(1,1)$ tensor field on a manifold $M$. Then the contracted vector field $C(T)\in \chi (T^*M) $ is defined at any point $(x,\omega )\in T^*M$ by its value on any evaluation function as follows:
\begin{equation}\label{3.2} 
C(T)(Z^V)_{(x,\omega )}=(TZ)^V_{(x,\omega )}=\omega ((TZ)_x), \quad Z\in \chi (M).
\end{equation}
For an 1-form $\alpha $ on $M$ we denote by $i_\alpha (T)$ the 1-form on $M$, defined by
\begin{equation}\label{3.3} 
(i_\alpha (T))(Z)=\alpha (TZ), \quad Z\in \chi (M).
\end{equation}
By using \eqref{3.3} we obtain
\begin{equation}\label{3.4}
(i_\alpha (T))^V(Z)^V_{(x,\omega )}=(\omega (T))^V(Z)^V_{(x,\omega )}=\omega ((TZ)_x), \quad Z\in \chi (M).
\end{equation}
Now, the equalities \eqref{3.2}, \eqref{3.4} and \propref{Proposition 2.1} imply that at each point 
$(x,\omega )\in T^*M$ the following equality holds
\begin{equation}\label{3.5}
C(T)_{(x,\omega )}=(\omega _x(T))^V .
\end{equation}
Also, at each point $(x,\omega )\in T^*M$ we have 
\begin{equation}\label{3.6}
W_{(x,\omega )}=(\omega _x)^V .
\end{equation}
Taking into account \eqref{3.1}, \eqref{3.5} and \eqref{3.6} we obtain
\begin{equation}\label{3.7}
\begin{array}{ll}
\overline g_{(x,\omega)}(X^C,C(T))=a\omega _x((TX)_x) , \quad \overline g_{(x,\omega)}(W,\alpha ^V)=0,
\\ \\
\overline g_{(x,\omega)}(W,W)=\overline g_{(x,\omega)}(W,C(T))=\overline g_{(x,\omega)}
(C(T_1),C(T_2))=0 ,
\end{array}
\end{equation}
where $T_1$ and $T_2$ are arbitrary $(1,1)$ tensor fields on $M$.
\par
For the Levi-Civita connection $\overline \nabla $ of the proper natural Riemann extension $\overline g$ we get the formulas (see \cite{KS}):
\begin{equation}\label{3.8}
\begin{array}{lllll}
(\overline \nabla _{X^C}Y^C)_{(x,\omega)}=(\nabla _XY)^C_{(x,\omega)}+C((\nabla X)(\nabla Y)+(\nabla Y)(\nabla X))_{(x,\omega )}\\ \\
\qquad \qquad \qquad \, +C(R(.,X)Y+R(.,Y)X)_{(x,\omega )}\\
\qquad \qquad \qquad \,\displaystyle -\frac{b}{2a}\left\{\omega (Y)X^C+\omega (X)Y^C+2\omega (Y)
C(\nabla X)+2\omega (X)C(\nabla Y)\right. \\
\qquad \qquad \qquad \, \, \left.\displaystyle +\omega (\nabla _XY+\nabla _YX)W\right\}_{(x,\omega )}+
\displaystyle \frac{b^2}{a^2}\omega (X)\omega (Y)W_{(x,\omega )}, \\ \\
(\overline \nabla _{X^C}\beta ^V)_{(x,\omega)}=(\nabla _X\beta )^V_{(x,\omega)}+
\displaystyle \frac{b}{2a}\left\{\omega (X)\beta ^V+\beta (X)W\right\}_{(x,\omega)}, \\ \\
(\overline \nabla _{\alpha ^V}Y^C)_{(x,\omega)}=-(i_\alpha (\nabla Y))^V_{(x,\omega)}+
\displaystyle \frac{b}{2a}\left\{\omega (Y)\alpha ^V+\alpha (Y)W\right\}_{(x,\omega)}, \\ \\
(\overline \nabla _{\alpha ^V}\beta ^V)_{(x,\omega)}=0 , \qquad \quad (\overline \nabla _{X^C}W)_{(x,\omega)}=-C(\nabla X)_{(x,\omega)}+\displaystyle \frac{b}{a}\omega (X)W_{(x,\omega)}, \\ \\
(\overline \nabla _{\alpha ^V}W)_{(x,\omega)}=\alpha ^V_{(x,\omega)}, \qquad \quad
(\overline \nabla _WW)_{(x,\omega)}=W_{(x,\omega)},
\end{array}
\end{equation}
where: $X^C, Y^C$ and  $\alpha ^V, \beta ^V$ are the complete lifts of the vector fields $X, Y\in \chi (M)$ and  the vertical lifts of the differential 1-forms $\alpha , \beta$ on $M$, respectively; 
$C(\nabla X)\in \chi (T^*M)$ is the contracted $(1,1)$ tensor field $\nabla X$ on $M$, defined by $(\nabla X)(Z)=\nabla _ZX$, $Z\in \chi (M)$; $R$ is the curvature tensor of $\nabla $ and $C(R(.,X)Y)$ is  the contracted $(1,2)$ tensor field $R(.,X)Y$ on $M$ given by $(R(.,X)Y)(Z)=R(Z,X)Y)$, 
$Z\in \chi (M)$.
\par
On $T^*M$ endowed with a proper natural Riemann extension $\overline g$ we define the endomorphism $P$ by
\begin{equation}\label{3.9}
\begin{array}{ll}
PX^C=X^C+2C(\nabla X)-\frac{b}{a}X^VW , \\
P\alpha ^V=-\alpha ^V .
\end{array}
\end{equation}
\begin{thm}\label{Theorem 3.1}
Let the total space of the cotangent bundle $T^*M$ of an $n$-dimensional manifold $(M,\nabla )$ be endowed with the proper natural Riemann extension $\overline g$, defined by \eqref{3.1}, and the endomorphism $P$, defined by \eqref{3.9}. Then $(T^*M,P,\overline g)$ is an almost para-Hermitian manifold. Moreover
\par
(i) if $M$ is not flat (resp. $M$ is flat), then $(T^*M,P,\overline g)$ is almost para-K\"ahler (resp.  
para-K\"ahler);
\par
(ii) $P$ is harmonic on $(T^*M,\overline g)$.
\end{thm}
\begin{proof}
From \eqref{3.5}, \eqref{3.6} and \eqref{3.9} it follows that 
\begin{equation}\label{3.10}
P(C(\nabla X))=-C(\nabla X) , \quad PW=-W .
\end{equation}
By using \eqref{3.9} and \eqref{3.10} we see that $P\neq {\rm Id}$ and $P^2={\rm Id}$. One can easily  verify that the eigendistributions of $P$ corresponding to the eigenvalues $\pm 1$ of $P$ have the same rank. Hence, $P$ is an almost para-complex structure on $T^*M$. By direct calculations, using \eqref{3.1},  \eqref{3.7} and \eqref{3.9} we obtain
\begin{equation*}
\begin{array}{ll}
\overline g(PX^C,PY^C)=-\overline g(X^C,Y^C), \quad \overline g(PX^C,P\alpha ^V)=-\overline g(X^C,\alpha ^V), \\
\overline g(P\alpha ^V,P\beta ^V)=-\overline g(\alpha ^V,\beta ^V),
\end{array}
\end{equation*}
which means that $(T^*M,P,\overline g)$ is an almost para-Hermitian manifold.
\par (i) Further, we find the tensor field 
$\overline F(\overline X,\overline Y,\overline Z)=\overline g((\overline \nabla _{\overline X}P)\overline Y,\overline Z)$ on  $(T^*M,P,\overline g)$, where $\overline X, \overline Y, \overline Z \in \chi (T^*M)$.
By using \eqref{3.1}, \eqref{3.7},  \eqref{3.8}, \eqref{3.9} and \eqref{3.10} we obtain
\begin{equation}\label{3.11}
\begin{array}{llll}
\overline F_{(x,\omega )}(X^C,Y^C,Z^C)=2a\omega (R_x(Z,Y)X),\\
\overline F_{(x,\omega )}(X^C,\alpha ^V,Z^C)=-\overline F_{(x,\omega )}(X^C,Z^C,\alpha ^V)=0, \\
\overline F_{(x,\omega )}(\alpha ^V,\beta ^V,Z^C)=-\overline F_{(x,\omega )}(\alpha ^V,,Z^C,\beta ^V)=0, \\
\overline F_{(x,\omega )}(\alpha ^V,Y^C,Z^C)=
\overline F_{(x,\omega )}(X^C,\beta ^V,\gamma ^V)=\overline F_{(x,\omega )}(\alpha ^V,\beta ^V,\gamma ^V)=0 .
\end{array}
\end{equation}
If $M$ is flat, then from \eqref{3.11} it follows that $\overline F(X^C+\alpha ^V,Y^C+\beta ^V,Z^C+\gamma ^V)=0$ for arbitrary $X^C+\alpha ^V,  Y^C+\beta ^V, Z^C+\gamma ^V\in \chi (T^*M)$ which means that $(T^*M,P,\overline g)$ is para-K\"ahler. 
In the case when $M$ is not flat, then the equalities \eqref{3.11} and the first identity of Bianchi for $R$ imply
\begin{equation}\label{3.12}
\begin{array}{ll}
\mathop{\s} \limits_{(X^C+\alpha ^V,Y^C+\beta ^V,Z^C+\gamma ^V)}\overline F(X^C+\alpha ^V,Y^C+\beta ^V,Z^C+\gamma ^V)=\\
\mathop{\s} \limits_{(X^C,Y^C,Z^C)}\overline F(X^C,Y^C,Z^C)=2a\mathop{\s} \limits_{(X,Y,Z)}\omega (R(Z,Y)X)=0 ,
\end{array}
\end{equation}
i.e. $(T^*M,P,\overline g)$ is an almost para-K\"ahler manifold.
\par (ii) As an consequence from the characteristic condition $\mathop{\s} \limits_{(X,Y,Z)}F(X,Y,Z)=0$ of an almost  para-K\"ahler manifold $(N,P,g)$ and the properties \eqref{3.0} of $F$ we obtain 
\begin{equation*}
F(PX,PY,Z)=F(X,Y,Z) , \quad X, Y, Z \in \chi (N).
\end{equation*}
The last equality implies $(\nabla _XP)Y=(\nabla _{PX}P)PY$.  Then if $\{e_1,\ldots ,e_n,Pe_1,\ldots ,Pe_n\}$ is an orthonormal basis on $N$, such that $g(e_i,e_i)=-g(Pe_i,Pe_i)=1$ $(i=1,\ldots ,n)$, 
for $\delta P$ we have 
\begin{equation*}
\delta P={\rm trace}_g\nabla P=\sum_{i=1}^{n}\left\{(\nabla _{e_i}P)e_i-(\nabla _{Pe_i}P)Pe_i\right\}=0 .
\end{equation*}
Hence,  the almost para-complex structure $P$ is harmonic on every  almost para-K\"ahler manifold 
$(N,P,g)$. In the case when $(N,P,g)$ is para-K\"ahler, then $\nabla P=0$ and $\delta P=0$ holds too.
\end{proof}
Now, let us assume that $T^*M$ is endowed with a non-proper natural Riemann extension $\overline g_1$, i.e. $\overline g_1$ is given by \eqref{3.1} and $b=0$.  We define the endomorphism $P_1$ by
\begin{equation}\label{3.13}
\begin{array}{ll}
P_1X^C=X^C+2C(\nabla X) , \\
P_1\alpha ^V=-\alpha ^V .
\end{array}
\end{equation}
By direct verification we establish that $(P_1,\overline g_1)$ is an almost para-Hermitian structure on 
$T^*M$ which is obtained from the almost para-Hermitian structure $(P,\overline g)$ on $T^*M$ by $b=0$. Moreover, from \eqref{3.11} we see that the tensor $\overline F$ on $(T^*M,P,\overline g)$ does not depend on $b$. Therefore we obtain
\begin{thm}\label{Theorem 3.2}
Let the total space of the cotangent bundle $T^*M$ of an $n$-dimensional manifold $(M,\nabla )$ be endowed with the non-proper natural Riemann extension $\overline g_1$ and the endomorphism $P_1$, defined by \eqref{3.13}. Then $(T^*M,P_1,\overline g_1)$ is an almost para-Hermitian manifold. Moreover
\par
(i)  if $M$ is not flat (resp. $M$ is flat), then $(T^*M,P_1,\overline g_1)$ is almost para-K\"ahler (resp.  
para-K\"ahler);
\par
(ii) $P_1$ is harmonic on $(T^*M,\overline g_1)$.
\end{thm}

\section{Almost paracontact metric structures induced by proper natural Riemann extensions}\label{sec-4}
In this section we will construct almost paracontact metric structures on  hypersurfaces of almost 
para-K\"ahler  and para-K\"ahler manifolds $(T^*M,P,\overline g)$ considered in \secref{sec-3}.
\par
A (2n+1)-dimensional smooth manifold $\widetilde M$
has an \emph{almost paracontact structure} $(\varphi,\overline \xi,\eta)$ if it admits a tensor field
$\varphi$ of type $(1,1)$, a vector field $\overline \xi$ and a 1-form
$\eta$ satisfying the  following  conditions:
\[
 \varphi^2 = {\rm Id} - \eta \otimes \overline \xi, \quad \eta (\overline \xi)=1, \quad 
 \varphi(\overline \xi)=0.
\] 
As immediate consequences of the definition of the almost paracontact structure we have that the endomorphism $\varphi$ has rank $2n$ and $\eta \circ \varphi=0$. If a manifold $\widetilde M$ with 
$(\varphi,\overline \xi,\eta)$-structure admits a pseudo-Riemannian metric $g$ such that
\[
g(\varphi X,\varphi Y)=-g(X,Y)+\eta (X)\eta (Y), \quad X, Y  \in \chi(\widetilde M)
\]
then we say that $\widetilde M$ has an almost paracontact metric structure and $(\widetilde M,\varphi,\overline \xi,\eta ,g)$ is called {\it an almost paracontact metric manifold}. The metric $g$ is called \emph{compatible} metric and it is necessarily of signature $(n+1,n)$. Setting $Y=\overline \xi$, we have $\eta(X)=g(X,\overline \xi)$.

The fundamental 2-form $\phi$  on $(\widetilde M,\varphi,\overline \xi,\eta ,g)$ is given by 
$\phi(X,Y)=g(X,\varphi Y)$
and the tensor field $\widetilde F$ of type $(0,3)$ is defined by
\[
\widetilde F(X,Y,Z)=(\widetilde \nabla \phi )(X,Y,Z)=(\widetilde \nabla _X\phi )(Y,Z)=\\
g((\widetilde \nabla _X\varphi )Y,Z) ,
\]
where $X, Y, Z \in \chi(\widetilde M)$ and $\widetilde \nabla $ is the Levi-Civita connection on 
$\widetilde M$. The tensor field $\widetilde F$ has the following properties:
\begin{equation*}
\begin{array}{ll}
\widetilde F(X,Y,Z)=-\widetilde F(X,Z,Y), \\
\widetilde F(X,\varphi Y, \varphi Z)=\widetilde F(X,Y,Z)+\eta(Y)\widetilde F(X,Z,\overline \xi)-\eta(Z)\widetilde F(X,Y,\overline \xi).
\end{array}
\end{equation*}
The following 1-forms are associated with $\widetilde F$:
\begin{equation*}
\theta(X)=g^{ij}\widetilde F(e_i,e_j,X); \,
\theta^*(X)=g^{ij}\widetilde F(e_i,\varphi e_j,X); \,
\omega(X)=\widetilde F(\overline \xi,\overline \xi,X),
\end{equation*}
where $\{e_i,\overline \xi\}$ $(i=1,\ldots,2n)$ is a basis of $T\widetilde M$, and $(g^{ij})$ is the inverse matrix of $(g_{ij})$.
An almost paracontact metric manifold is called
\begin{itemize}
\item {\it normal} if $N(X,Y)-2d\eta (X,Y)\overline \xi = 0$, where
\[
N(X,Y)=\varphi ^2[X,Y]+[\varphi X,\varphi Y]-\varphi [\varphi X,Y]-\varphi [X,\varphi Y]
\]
is the Nijenhuis torsion tensor of $\varphi $ (see \cite{});
\item {\it paracontact metric} if $\phi =d\eta$;
\item  {\it $\alpha $-para-Sasakian} if $(\widetilde \nabla_X\varphi)Y=\alpha(g(X,Y)\overline \xi-\eta(Y)X)$, where $\alpha\neq 0$ is constant; \item {\it para-Sasakian} if it is normal and paracontact metric;
\item  {\it $\alpha $-para-Kenmotsu} if $(\widetilde \nabla_X\varphi)Y=-\alpha(g(X,\varphi Y)\overline \xi+\eta(Y)\varphi X)$, where $\alpha\neq 0$ is constant, in particular, para-Kenmotsu if $\alpha=-1$;
\item {\it K-paracontact} if it is paracontact and $\overline \xi$ is Killing vector field;
\item {\it quasi-para-Sasakian} if it is normal and $d\phi =0$.
\end{itemize}
Twelve basic classes of almost paracontact metric manifolds with respect to the tensor field  $\widetilde F$ were obtained in \cite{ZN}. Further we give the characteristic conditions of these classes:
\begin{equation}\label{4.101}
\begin{array}{ll}
\mathbb{G}_1:\widetilde F(X,Y,Z)=\displaystyle\frac{1}{2(n-1)}\{g(X,\varphi Y)\theta (\varphi Z)-
g(X,\varphi Z)\theta (\varphi Y)\\ \\
\qquad \qquad \qquad \quad -g(\varphi X,\varphi Y)\theta (\varphi ^2Z)+g(\varphi X,\varphi Z)\theta (\varphi ^2Y)\} ,
\end{array}
\end{equation}
\begin{equation}\label{4.102}
\begin{array}{l}
\mathbb{G}_2:\widetilde F(\varphi X ,\varphi Y,Z)=-\widetilde F(X,Y,Z) , \qquad  \theta =0 ,
\end{array}
\end{equation}
\begin{equation}\label{4.103}
\mathbb{G}_3:\widetilde F(\overline \xi ,Y,Z)=\widetilde F(X,\overline \xi ,Z)=0, \qquad  \widetilde F(X,Y,Z)=-\widetilde F(Y,X,Z) ,
\end{equation}
\begin{equation}\label{4.104}
\mathbb{G}_4:\widetilde F(\overline \xi ,Y,Z)=\widetilde F(X,\overline \xi ,Z)=0, \quad   \mathop{\s} \limits_{(X,Y,Z)}\widetilde F(X,Y,Z)=0 ,
\end{equation}
\begin{equation}\label{4.105}
\mathbb{G}_5:\widetilde F(X,Y,Z)=\displaystyle{\frac{\theta (\overline \xi)}{2n}}\{\eta(Y)g(\varphi X,\varphi Z)-\eta(Z)g(\varphi X,\varphi Y)\} ,
\end{equation}
\begin{equation}\label{4.106}
\mathbb{G}_6:\widetilde F(X,Y,Z)=-\displaystyle{\frac{\theta ^*(\overline \xi)}{2n}}\{\eta(Y)g(X,\varphi Z)-\eta(Z)g(X,\varphi Y)\} ,
\end{equation}
\begin{equation}\label{4.107}
\begin{array}{ll}
\mathbb{G}_7:\widetilde F(X,Y,Z)=-\eta(Y)\widetilde F(X,Z,\overline \xi )+\eta(Z)\widetilde F(X,Y,\overline\xi ),\\  \\
\qquad \qquad \widetilde F(X,Y,\overline \xi )=-\widetilde F(Y,X,\overline \xi )=-\widetilde F(\varphi X,\varphi Y,\overline \xi ), \quad \theta ^*(\overline \xi)=0 , 
\end{array}
\end{equation}
\begin{equation}\label{4.108}
\begin{array}{ll}
\mathbb{G}_8:\widetilde F(X,Y,Z)=-\eta(Y)\widetilde F(X,Z,\overline \xi )+\eta(Z)\widetilde F(X,Y,\overline \xi ),
\\  \\
\qquad \qquad \widetilde F(X,Y,\overline \xi )=\widetilde  F(Y,X,\overline \xi )=-\widetilde  F(\varphi X,\varphi Y,\overline \xi ), \quad \theta (\overline \xi)=0 ,
\end{array}
\end{equation}
\begin{equation}\label{4.109}
\begin{array}{ll}
\mathbb{G}_9: \widetilde F(X,Y,Z)=-\eta(Y)\widetilde F(X,Z,\overline \xi )+\eta(Z)\widetilde F(X,Y,\overline \xi ), \\ \\
\qquad \qquad \widetilde F(X,Y,\overline \xi )=-\widetilde F(Y,X,\overline \xi )=\widetilde F(\varphi X,\varphi Y,\overline \xi ) , 
\end{array}
\end{equation}
\begin{equation}\label{4.110}
\begin{array}{ll}
\mathbb{G}_{10}:\widetilde F(X,Y,Z)=-\eta(Y)\widetilde F(X,Z,\overline \xi )+\eta(Z)\widetilde F(X,Y,\overline \xi ), \\ \\
\qquad \qquad \widetilde F(X,Y,\overline \xi )=\widetilde F(Y,X,\overline \xi )=\widetilde F(\varphi X,\varphi Y,\overline \xi ),
\end{array}
\end{equation}
\begin{equation}\label{4.111}
\mathbb{G}_{11}:\widetilde  F(X,Y,Z)=\eta(X)\widetilde F(\overline \xi ,\varphi Y,\varphi Z) , 
\end{equation}
\begin{equation}\label{4.112}
\mathbb{G}_{12}:\widetilde F(X,Y,Z)=\eta(X)\left\{\eta(Y)\widetilde F(\overline \xi ,\overline \xi ,Z)-\eta(Z)\widetilde F(\overline \xi ,\overline \xi ,Y)\right\} .
\end{equation}
In \cite{ZN} the classes of $\alpha$-para-Sasakian,  $\alpha$-para-Kenmotsu, normal, paracontact metric, para-Sasakian, K-paracontact and quasi-para-Sasakian manifolds are determined. Also, the classes of the 3-dimensional almost paracontact metric manifolds are obtained. Here, we recall some of the theorems in \cite{ZN} which we need. 
\par
Let $\overline {\mathbb{G}}_5$ be the subclass of $\mathbb{G}_5$ which consists of all $(2n+1)$-dimensional $\mathbb{G}_5$-manifolds such that $\theta (\xi )=2n$
(resp. $\theta (\xi )=-2n$) by $\phi (X,Y)=g(\varphi X,Y)$ (resp. $\phi (X,Y)=g(X,\varphi Y)$).
\begin{thm}\label{Theorem A}\cite{ZN}
A $(2n+1)$-dimensional almost paracontact metric manifold $(\widetilde M,\varphi ,\overline \xi ,\\
\eta ,g)$ is: 
\par
(i) paracontact metric if and only if $\widetilde M$ belongs to  the class $\overline {\mathbb{G}}_5$  or to the classes which are direct sums of $\overline {\mathbb{G}}_5$ with $\mathbb{G}_4$ and $\mathbb{G}_{10}$;
\par
(ii) para-Sasakian if and only if $\widetilde M$ belongs to  the class $\overline {\mathbb{G}}_5$;
\par
(iii) K-paracontact metric if and only if $\widetilde M$ belongs to the classes $\overline {\mathbb{G}}_5$ and  $\overline {\mathbb{G}}_5\oplus \mathbb{G}_4$;
\par
(iv) quasi-para-Sasakian if and only if $\widetilde M$ belongs to the classes $\mathbb{G}_5$, $\mathbb{G}_8$ and \\ $\mathbb{G}_5\oplus \mathbb{G}_8$.
\end{thm}
\begin{pro}\label {Proposition B}\cite{ZN}
The 3-dimensional almost paracontact metric manifolds belong to the classes $\mathbb{G}_5$,
$\mathbb{G}_6$, $\mathbb{G}_{10}$, $\mathbb{G}_{12}$ and to the classes which are their direct sums.
\end{pro}

Let $(\overline M,P,\overline g)$ be a $2n$-dimensional almost para-Hermitian manifold and $\widetilde M$ be a $(2n-1)$-dimensional differentiable hypersurface embeding in $\overline M$ such that the normal vector field $N$ to $\widetilde M$ is a time-like unit, i.e. $\overline g(N,N)=-1$. Hence, $PN$ is a space-like unit tangent vector field on $\widetilde M$. We denote the tangent and the normal component of the transform vector field $PX$ of an arbitrary tangent vector field $X\in \chi(\widetilde M)$ by $\varphi X$ and
$\eta (X)N$, respectively. Then $PX\in \chi (\widetilde M)$ has the unique decomposition $PX=\varphi X+\eta (X)N$, where $\varphi $ is an $(1,1)$ tensor field on $\widetilde M$.  The 1-form $\eta$  on 
$\widetilde M$ is defined by $\eta (X)=\overline g(X,PN)$. So, at every point $p\in \widetilde M$ is determined the structure $(\varphi ,\overline \xi ,\eta ,g)$, where
\begin{equation}\label{4.1} 
\varphi X=PX-\eta (X)N, \quad \overline \xi=PN, \quad \eta (X)=\overline g(X,PN), \quad X\in \chi (\widetilde M)
\end{equation}
and by $g$ is denoted the restriction of $\overline g$ on $\widetilde M$. It is easy to check that 
$(\varphi ,\overline \xi ,\eta ,g)$ is an almost paracontact metric structure on $\widetilde M$, i.e. 
$(\widetilde M,\varphi ,\overline \xi ,\eta ,g)$ is a $(2n-1)$-dimensional almost  paracontact metric manifold. 
\par
Let $\overline \nabla$ and $\widetilde \nabla $ be the Levi-Civita connections of the metrics $\overline g$ and $g$ on $\overline M$ and $\widetilde M$, respectively. Then the formulas of Gauss and Weingarten are:
\begin{equation}\label{4.2} 
\overline \nabla _XY=\widetilde \nabla _XY-g(A_NX,Y)N , \quad \overline \nabla _XN=-A_NX, \quad X, Y \in \chi (\widetilde M),
\end{equation}
where $A_N$ is the second fundamental tensor of $\widetilde M$ corresponding to $N$.
\par
Using \eqref{4.1} and  \eqref{4.2} we obtain
\begin{equation}\label{4.3} 
\overline F(X,Y,Z)=\widetilde F(X,Y,Z)-\eta (Y)g(A_NX,Z)+\eta (Z)g(A_NX,Y) ,
\end{equation}
\begin{equation}\label{4.4}
\overline F(X,Y,N)=\widetilde F(X,\varphi Y,\overline \xi )+g(A_NX,\varphi Y) ,
\end{equation}
where $X,Y,Z\in \chi (\widetilde M)$ and $\overline F$, $\widetilde F$ are the tensor fields on $\overline M$ and $\widetilde M$, defined by $\overline F(X,Y,Z)=\overline g((\overline \nabla _XP)Y,Z)$, \,
$\widetilde F(X,Y,Z)=g((\widetilde \nabla _X\varphi )Y,Z)$, respectively.
Let us assume that the $n$-dimensional manifold $M$ is endowed with both a symmetric linear connection 
$\nabla $ and with a globally defined nowhere zero vector field $\xi $ which is parallel with respect to 
$\nabla $, i.e. $\nabla \xi =0$ and $f$ is a function on $M$. 
\par
We consider the function $\widetilde  f: T^*M\longrightarrow  \mathbb{R}$ defined by
\[
\widetilde f=\xi ^V+f^V ,
\]
or equivalently by $\widetilde f(x,\omega )=\omega _x(\xi _x)+f(x)$ for any $(x,\omega )\in T^*M$.
\par
Let 
\[
\widetilde H_t=\widetilde f^{-1}(t)=\{(x,\omega )\in T^*M : \widetilde f(x,\omega )=
t, \, t\in \mathbb{R}\}
\]
be the hypersurfaces level set in $T^*M$, endowed with the restriction $g$  of the proper natural Riemann extension $\overline g$ on $T^*M$, where $f(x)\neq t$ at any point $x$ in $M$.
\par
For later use, we recall that the gradient of a real function $F : N\longrightarrow  \mathbb{R}$ on a (semi-)
Riemannian manifold $(N,h)$ is given by $h({\rm grad}F,X)={\rm d}F(X)$, $X\in \chi (N)$ and $h$ is a (semi-) Riemannian metric on $N$. In \cite{BK} the following formula for the gradient of the vertical lift $Z^V$ on $T^*M$ of $Z\in \chi (M)$ with respect to the proper natural Riemann extension $\overline g$ on 
$T^*M$ is obtained:
\begin{equation}\label{4.8}
{\rm grad}Z^V=\frac{1}{a}\left\{Z^C+2C(\nabla Z)-\frac{b}{a}Z^VW\right\} .
\end{equation}
\begin{thm}\label{Theorem 4.2}
Let $(M,\nabla )$ be a manifold endowed with a symmetric linear connection $\nabla $ inducing the proper natural  Riemann extension $\overline g$ on $T^*M$ and $f$ be a function on $M$. If $t\in \mathbb{R}$ and $f(x)\neq t$ at any point $x$ in $M$, then:
\par
(i) At any point $(x,\omega )$ of $\widetilde H_t$ the gradient of the function  $\widetilde f$ is a normal vector field to $\widetilde H_t$ and it is given by
\begin{equation}\label{4.9}
{\rm grad}\widetilde f=\frac{1}{a}\left\{\xi ^C-\frac{b}{a}\xi ^VW+({\rm d}f)^V\right\} .
\end{equation}
\par
(ii) The restriction $g$ of $\overline g$ on $\widetilde H_t$ is non-degenerate on $\widetilde H_t$, i.e. 
$(\widetilde H_t,g)$ is a semi-Riemannian hypersurface of $T^*M$.
\par
(iii) The vertical lift $\alpha ^V$ of an 1-form $\alpha $ on $M$ and the complete lift $X^C$ of $X\in \chi (M)$ are tangent to $\widetilde H_t$ if at any point $(x,\omega )\in \widetilde H_t$ they satisfy the conditions:
\begin{equation}\label{4.10}
\alpha _x(\xi _x)=0 , \quad (Xf)(x)=\omega _x((\nabla _\xi X)_x) .
\end{equation}
\end{thm}
\begin{proof}
(i) By using $\overline g({\rm grad}\widetilde f,\overline U)=({\rm d}\widetilde f)(\overline U)$ for any tangent vector field $\overline U$ on $T^*M$ and $\widetilde f(x,\omega )=t\in \mathbb{R}$ at any point 
$(x,\omega )\in \widetilde H_t$, we obtain that $\overline g({\rm grad}\widetilde f,U)=0$ for any vector field $U$ on $\widetilde H_t$. Therefore, ${\rm grad}\widetilde f$ is a normal vector field to 
$\widetilde H_t$.
\par
From the definition of the function $\widetilde f$ it follows that ${\rm grad}\widetilde f={\rm grad}\xi ^V+{\rm grad}f^V$. For ${\rm grad}\xi ^V$, using \eqref{4.8} and taking into account that $\nabla \xi =0$, we have
\[
{\rm grad}\xi ^V=\frac{1}{a}\left\{\xi ^C-\frac{b}{a}\xi ^VW\right\} .
\]
Now, let us assume that ${\rm grad}f^V=Y^C+\theta ^V$, where $Y\in \chi (M)$ and $\theta $ is an 1-form on $M$. Substituting ${\rm grad}f^V=Y^C+\theta ^V$ in the equality $\overline g({\rm grad}f^V,\alpha ^V)=\alpha ^V(f^V)=0$ we obtain $a\alpha (Y)=0$ for any 1-form $\alpha $ on $M$, which implies $Y=0$. Then from $\overline g({\rm grad}f^V,X^C)=\overline g(\theta ^V,X^C)=a\theta (X)$ and
$\overline g({\rm grad}f^V,X^C)=X^C(f^V)=(Xf)^V=(({\rm d}f)X)^V$ it follows that 
$\theta =\frac{1}{a}{\rm d}f$. Hence, ${\rm grad}f^V=\frac{1}{a}({\rm d}f)^V$ and \eqref{4.9} holds.
\par
(ii) For the normal vector field ${\rm grad}\widetilde f$ to $\widetilde H_t$ we compute $\overline g({\rm grad}\widetilde f,{\rm grad}\widetilde f)=-\frac{b(\omega (\xi ))^2}{a^2}$, which shows that ${\rm grad}\widetilde f$ is time-like or space-like when $b>0$ or $b<0$, respectively. Consequently, (ii) is proved.
\par
(iii)  $\alpha ^V$ and $X^C$ are tangent to $\widetilde H_t$ if at any point $(x,\omega )\in \widetilde H_t$
$\overline g({\rm grad}\widetilde f,X^C)=\overline g({\rm grad}\widetilde f,\alpha ^V)=0$. By using \eqref{4.9} we obtain the equalities in \eqref{4.10}.
\end{proof}
Further, we consider a hypersurface $\widetilde H_t$ of $(T^*M,P,\overline g)$ with a time-like unit normal vector field $N$. According to \thmref{Theorem 4.2},  ${\rm grad}\widetilde f$ is a normal vector field to 
$\widetilde H_t$ and it is time-like if $b>0$. Hence,
\begin{equation}\label{4.11}
N=\frac{1}{\sqrt{b}\xi ^V}\left\{\xi ^C-\frac{b}{a}\xi ^VW+({\rm d}f)^V\right\} .
\end{equation}
Supplying  $\widetilde H_t$ with the almost paracontact metric structure defined by \eqref{4.1}, we have:
\begin{equation}\label{4.12}
\begin{array}{lll}
\displaystyle \overline \xi =\frac{1}{\sqrt{b}\xi ^V}\left\{\xi ^C-({\rm d}f)^V\right\}, \quad \eta (X^C)=-\frac{2a}{\sqrt{b}\xi ^V}(Xf)^V+\sqrt{b}X^V,  \\ \\
\displaystyle \varphi X^C=X^C+2C(\nabla X)-\frac{2a}{\sqrt{b}\xi ^V}(Xf)^VW-\eta (X^C)\frac{1}{\sqrt{b}\xi ^V}\left\{\xi ^C+({\rm d}f)^V\right\} , \\
\eta (\alpha ^V)=0, \qquad  \varphi \alpha ^V=-\alpha ^V .
\end{array}
\end{equation}
\begin{thm}\label{Theorem 4.3}
For the $(2n-1)$-dimensional almost paracontact metric manifold $(\widetilde H_t,\varphi ,\overline \xi ,\eta ,g)$ of $(T^*M,P,\overline g)$ with a time-like unit normal vector field $N$ and an almost paracontact metric structure given by \eqref{4.11}  and \eqref{4.12}, respectively, we have:
\par
(i) If $M$ is flat or ${\rm dim}M=2$, then $\widetilde H_t \in \mathbb{G}_5\oplus \mathbb{G}_{10}$.
\par
(ii) If $M$ is not flat and ${\rm dim}M>2$, then $\widetilde H_t \in \mathbb{G}_4\oplus \mathbb{G}_5\oplus \mathbb{G}_{10}$.\\
In both cases (i) and (ii) $\widetilde H_t$ is paracontact metric if and only if $b=4a^2$.
\end{thm}
\begin{proof}
From \eqref{4.3} for the tensor $\widetilde F$ on $\widetilde H_t$ we have
\begin{equation}\label{4.13}
\widetilde F(\widetilde X,\widetilde Y,\widetilde Z)=\overline F(\widetilde X,\widetilde Y,\widetilde Z)+\eta (\widetilde Y)g(A_N\widetilde X,\widetilde Z)-\eta (\widetilde Z)g(A_N\widetilde X,\widetilde Y), 
\end{equation}
where $\widetilde X,\widetilde Y,\widetilde Z \in \chi (\widetilde H_t)$.\\
For arbitrary $X^C\in \chi (\widetilde H_t)$ and  $\alpha ^V\in \chi (\widetilde H_t)$, using \eqref{3.8}, we find
\begin{equation}\label{4.14}
\begin{array}{ll}
\displaystyle A_NX^C=-\overline \nabla _{X^C}N=-\frac{1}{\sqrt{b}\xi ^V}\left\{C(R(.,\xi )X)+(\nabla _X{\rm d}f)^V\right\} \\
\displaystyle \qquad \qquad \qquad \qquad \, \, \, +\frac{\sqrt{b}}{2a}\left\{X^C+\eta (X^C)\overline \xi \right\}-
\displaystyle \frac{2(\nabla _\xi X)^V}{\sqrt{b}(\xi ^V)^2}({\rm d}f)^V,
\end{array}
\end{equation}
\begin{equation}\label{4.15}
A_N\alpha ^V=-\overline \nabla _{\alpha ^V}N=\frac{\sqrt{b}}{2a}\alpha ^V .
\end{equation}
Next, we calculate 
\begin{equation}\label{4.16}
\begin{array}{ll}
\qquad \qquad \qquad \, \, \displaystyle \overline g_{(x,\omega )}(A_NX^C,Z^C)=-\frac{a}{\sqrt{b}\omega (\xi )}\left\{\omega (R(Z,\xi )X)+X(Zf)\right. \\
\displaystyle \qquad \qquad \qquad \left.-(\nabla _XZ)(f)\right\}_{(x,\omega )} 
+\frac{\sqrt{b}}{2a}\left\{\overline g(X^C,Z^C)+\eta (X^C)\eta (Z^C)\right\}_{(x,\omega )}
\\ \\
\displaystyle\qquad \qquad \qquad -\left\{\frac{2a\omega (\nabla _\xi X)}{\sqrt{b}(\omega (\xi ))^2}(Zf)\right\}_{(x,\omega )} .
\end{array}
\end{equation}
From the first identity of Bianchi and $\nabla \xi =0$ we get
\begin{equation}\label{4.17}
R(Z,\xi )X=R(X,\xi )Z , \quad X, \xi , Z\in \chi (M) .
\end{equation}
Since $C(\nabla Z)$ is a vertical vector field on $T^*M$ and for $X^C\in \chi (\widetilde H_t)$ the following equality 
\begin{equation}\label{4.18}
(\nabla _\xi X)^V=(Xf)^V
\end{equation}
holds, we obtain
\begin{equation}\label{4.19}
\omega _x\left(\left(\nabla _{\nabla _\xi X}Z\right)_x\right)=C(\nabla Z)\left((\nabla _\xi X)^V\right)_{(x,\omega )}=C(\nabla Z)\left((Xf)^V\right)_{(x,\omega )}=0 .
\end{equation}
The equalities \eqref{4.17} and \eqref{4.19} imply
\begin{equation}\label{4.20}
\omega _x(R_x(Z,\xi )X)=\omega _x\left((\nabla _X\nabla _\xi Z)_x\right)-\omega _x\left((\nabla _\xi \nabla _X Z)_x\right) .
\end{equation}
By using \eqref{4.18} and \eqref{4.19} we get 
\begin{equation}\label{4.21}
\begin{array}{ll}
(X(Zf))_x=X^C((Zf)^V)_{(x,\omega )}=X^C((\nabla _\xi Z)^V)_{(x,\omega )}=\left[X,\nabla _\xi Z\right]^V_{(x,\omega )}\\
\qquad \qquad \, \, \, =\omega _x\left((\nabla _X\nabla _\xi Z)_x\right) .
\end{array}
\end{equation}
Now, we substitute \eqref{4.21} and $\left((\nabla _XZ)(f)\right)_x=\omega _x\left((\nabla _\xi \nabla _X Z)_x\right)$ in \eqref{4.16}. Then, taking into account \eqref{4.10} and \eqref{4.20}, the equality \eqref{4.16} becomes
\[
\begin{array}{lll}
\displaystyle \overline g_{(x,\omega )}(A_NX^C,Z^C)=-\frac{2a}{\sqrt{b}\omega (\xi )}\left\{\omega (R(Z,\xi )X)\right\}_{(x,\omega )} \\
\displaystyle \qquad \qquad \qquad \qquad+\frac{\sqrt{b}}{2a}\left\{\overline g(X^C,Z^C)+\eta (X^C)\eta (Z^C)\right\}_{(x,\omega )}
\\ \\
\displaystyle\qquad \qquad \qquad \qquad -\left\{\frac{2a}{\sqrt{b}(\omega (\xi ))^2}(Xf)(Zf)\right\}_{(x,\omega )} .
\end{array}
\]
By using \eqref{3.11}, \eqref{4.13} and the above equality we obtain
\begin{equation}\label{4.22}
\begin{array}{lll}
\qquad \qquad \qquad \displaystyle\widetilde F_{(x,\omega )}(X^C,Y^C,Z^C)=\frac{2a}{\sqrt{b}\omega (\xi )}\left\{
\sqrt{b}\omega (\xi )\omega (R(Z,Y)X)\right. \\ \\
\qquad \qquad \qquad  \, \, \, \left.\displaystyle-\omega (R(Z,\xi )X)\eta (Y^C)
+\omega (R(Y,\xi )X)\eta (Z^C)\right\}_
{(x,\omega )}\\ \\
\qquad \qquad \qquad  \, \, \, \displaystyle +\frac{\sqrt{b}}{2a}\left\{\overline g(X^C,Z^C)\eta (Y^C)-\overline g(X^C,Y^C)\eta (Z^C)\right\}_{(x,\omega )} \\ \\
\qquad \qquad \qquad  \, \, \, \displaystyle +\frac{2a}{\sqrt{b}(\omega (\xi ))^2}\left\{
-(Xf)(Zf)\eta (Y^C)+(Xf)(Yf)\eta (Z^C)\right\}_{(x,\omega )} .
\end{array}
\end{equation}
After standard calculations , using \eqref{4.14} and \eqref{4.15}, we find
\begin{equation}\label{4.23}
\begin{array}{lll}
\displaystyle\widetilde F_{(x,\omega )}(X^C,Y^C,\gamma ^V)=\frac{\sqrt{b}}{2}\left(\gamma (X)\eta (Y^C)\right)_{(x,\omega )}, \\ \\
\displaystyle\widetilde F_{(x,\omega )}(\alpha ^V,Y^C,Z^C)=\frac{\sqrt{b}}{2}\left(\alpha (Z)\eta (Y^C)-
\alpha (Y)\eta (Z^C)\right)_{(x,\omega )}, \\ \\
\widetilde F_{(x,\omega )}(\alpha ^V,\beta ^V,Z^C)=\widetilde F_{(x,\omega )}(X^C,\beta ^V,\gamma ^V)=\widetilde F_{(x,\omega )}(\alpha ^V,\beta ^V,\gamma ^V)=0 .
\end{array}
\end{equation}
Finally, using \eqref{4.22} and \eqref{4.23}, we obtain
\begin{equation}\label{4.24}
\begin{array}{ll}
\widetilde F_{(x,\omega )}(X^C+\alpha ^V,Y^C+\beta ^V,Z^C+\gamma ^V) \\
=\left(\widetilde F^\prime +\widetilde F^{''}+\widetilde F^{'''}\right)_{(x,\omega )}(X^C+\alpha ^V,Y^C+\beta ^V,Z^C+\gamma ^V) ,
\end{array}
\end{equation}
where
\begin{equation}\label{4.25}
\begin{array}{ll}
\displaystyle\widetilde F^\prime _{(x,\omega )}(X^C+\alpha ^V,Y^C+\beta ^V,Z^C+\gamma ^V)=\frac{2a}{\sqrt{b}\omega (\xi )}\left\{\sqrt{b}\omega (\xi )\omega (R(Z,Y)X)\right. \\ \\
\qquad \qquad \qquad  \qquad \, \, \, \displaystyle \left.\displaystyle-\omega (R(Z,\xi )X)\eta (Y^C)
+\omega (R(Y,\xi )X)\eta (Z^C)\right\}_{(x,\omega )},
\end{array}
\end{equation}
\begin{equation}\label{4.26}
\begin{array}{ll}
\displaystyle \widetilde F^{''} _{(x,\omega )}(X^C+\alpha ^V,Y^C+\beta ^V,Z^C+\gamma ^V) \\
=\displaystyle -\frac{\sqrt{b}}{2a}\left\{-\eta (Y^C)\left[\overline g(X^C,Z^C)+a\alpha (Z)+a\gamma (X)\right]\right. \\ \\
\displaystyle \left.+\eta (Z^C)\left[\overline g(X^C,Y^C)+a\alpha (Y)+a\beta (X)\right]\right\}_{(x,\omega )},
\end{array}
\end{equation}
\begin{equation}\label{4.27}
\begin{array}{ll}
\displaystyle \widetilde F^{'''} _{(x,\omega )}(X^C+\alpha ^V,Y^C+\beta ^V,Z^C+\gamma ^V) \\
\displaystyle =\frac{2a}{\sqrt{b}(\omega (\xi ))^2}\left\{
-(Xf)(Zf)\eta (Y^C)+(Xf)(Yf)\eta (Z^C)\right\}_{(x,\omega )} .
\end{array}
\end{equation}
By direct calculations we verify that for $\widetilde F^\prime $, $\widetilde F^{''}$ and $\widetilde F^{'''}$ the conditions \eqref{4.104}, \eqref{4.105} and \eqref{4.110} hold, respectively.  
\par
(i) The assumption that $M$ is flat implies $\widetilde F^\prime =0$. If ${\rm dim}M=2$, then 
${\rm dim}\widetilde H_t=3$ and from \propref{Proposition B} it follows that $\widetilde F^\prime $ vanishes too. Hence, $\widetilde F=\widetilde F^{''}+\widetilde F^{'''}$ which means that $\widetilde H_t$ belongs to the class $\mathbb{G}_5\oplus \mathbb{G}_{10}$.
\par
(ii) In the case when $M$ is not flat and ${\rm dim}M>2$ we have $\widetilde F=\widetilde F^\prime 
+\widetilde F^{''}+\widetilde F^{'''}$. Therefore $\widetilde H_t \in \mathbb{G}_4\oplus \mathbb{G}_5\oplus \mathbb{G}_{10}$.
\par
According to the assertion (i) from \thmref{Theorem A}, $\widetilde H_t$ is paracontact metric in both cases (i) and (ii) if and only if $\mathbb{G}_5=\overline {\mathbb{G}}_5$. From \eqref{4.26} we find $\displaystyle \theta _{\widetilde F^{''}}(\overline \xi )=-\frac{(n-1)\sqrt{b}}{a}$. Taking into account  the definition of $\overline {\mathbb{G}}_5$ in the case when $\phi (X,Y)=g(X,\varphi Y)$, we conclude that $\widetilde F^{''}$ satisfies the characteristic condition of the class $\overline {\mathbb{G}}_5$ if and only if 
$\displaystyle -\frac{(n-1)\sqrt{b}}{a}=-2(n-1)$. The last equality is equivalent to $b=4a^2$, which completes the proof.
\end{proof}
Now, we consider the function $\bar f: T^*M\longrightarrow  \mathbb{R}$ defined in \cite{B2} by
\[
\bar f=\xi ^V ,
\]
or equivalently by $\bar f(x,\omega )=\omega _x(\xi _x)$ for any $(x,\omega )\in T^*M$.
\par
Let 
\[
H_t=\bar f^{-1}(t)=\{(x,\omega )\in T^*M : \bar f(x,\omega )=t, \, t\in \mathbb{R}\setminus \{0\}\}
\]
be the hypersurfaces level set in $T^*M$, endowed with the restriction $g$  of the proper natural Riemann extension $\overline g$ on $T^*M$.
\par
We note that the hypersurfaces level set $H_t$ in $T^*M$ defined in \cite{B2} is a particular case from the set $\widetilde H_t$ which is obtained by $f=const$. In \cite{B2} it is shown that:
\par 
(1) At any point $(x,\omega )$ of $H_t$ the gradient of the function  $\bar f$ is a normal vector field to 
$H_t$ and it is given by
\[
{\rm grad}\bar f=\frac{1}{a}\left\{\xi ^C-\frac{b}{a}\xi ^VW\right\} .
\]
\par
(2) The restriction $g$ of $\overline g$ on $H_t$ is non-degenerate on $H_t$, i.e. 
$(H_t,g)$ is a semi-Riemannian hypersurface of $T^*M$.
\par
(3) The vertical lift $\alpha ^V$ of an 1-form $\alpha $ on $M$ and the complete lift $X^C$ of $X\in \chi (M)$ are tangent to $H_t$ if at any point $(x,\omega )\in H_t$ they satisfy the conditions:
\begin{equation}\label{4.28}
\alpha _x(\xi _x)=0 , \qquad  \omega _x((\nabla _\xi X)_x)=0  .
\end{equation}
We remark that the above three results are immediate consequences from \thmref{Theorem 4.2}.
\par
From \eqref{4.11} we obtain that by $b>0$ the vector field $N$ given by 
\begin{equation}\label{4.29}
N=\frac{1}{\sqrt{b}\xi ^V}\left\{\xi ^C-\frac{b}{a}\xi ^VW\right\} 
\end{equation}
is a time-like unit normal vector field to $H_t$. We endow the hypersurface $H_t$ of $(T^*M,P,\overline g)$
with the almost paracontact metric structure defined by \eqref{4.1}.  By using \eqref{4.12} we get:
\begin{equation}\label{4.30}
\begin{array}{lll}
\displaystyle \overline \xi =\frac{1}{\sqrt{b}\xi ^V}\xi ^C, \qquad \eta (X^C)=\sqrt{b}X^V,  \qquad 
\eta (\alpha ^V)=0 \\ \\
\displaystyle \varphi X^C=X^C+2C(\nabla X)-\frac{X^V}{\xi ^V}\xi ^C , \qquad \varphi \alpha ^V=-\alpha ^V .
\end{array}
\end{equation}
\begin{thm}\label{Theorem 4.4}
For the $(2n-1)$-dimensional almost paracontact metric manifold $(H_t,\varphi ,\overline \xi ,\eta ,g)$ of 
$(T^*M,P,\overline g)$ with a time-like unit normal vector field $N$ and an almost paracontact metric structure given by \eqref{4.29}  and \eqref{4.30}, respectively, we have:
\par
(i) If $M$ is flat or ${\rm dim}M=2$, then $H_t \in \mathbb{G}_5$ and hence $H_t$ is quasi-para-Sasakian. In this case $H_t$ is para-Sasakian if and only if $b=4a^2$.
\par
(ii) If $M$ is not flat and ${\rm dim}M>2$, then $H_t \in \mathbb{G}_4\oplus \mathbb{G}_5$. In this case $H_t$ is K-paracontact metric if and only if $b=4a^2$.
\end{thm}
\begin{proof}
We find the tensor field $\widetilde F$ of $H_t$ by using \eqref{4.24}, \eqref{4.25}, \eqref{4.26} and
\eqref{4.27}, taking into account that $f=const$. For arbitrary $X^C\in \chi(H_t)$ the equality
\eqref{4.28} implies $\nabla _\xi X=0, \, X\in \chi (M)$. From the last equality and $\nabla \xi =0$ it follows that $R(Z,\xi )X=R(X,\xi )Z=0$, \, $X, \xi ,Z \in \chi (M)$. Then the tensor field $\widetilde F^\prime $, defined by \eqref{4.25}, becomes
\begin{equation}\label{4.31}
\widetilde F^\prime _{(x,\omega )}(X^C+\alpha ^V,Y^C+\beta ^V,Z^C+\gamma ^V)=2a\omega _x (R_x(Z,Y)X) .
\end{equation}
One can easily check that $\widetilde F^\prime $ given by \eqref{4.31} satisfies \eqref{4.104}. Since $f=const$ the tensor field $\widetilde F^{'''}$, defined by \eqref{4.27}, vanishes. Consequently, for the tensor field $\widetilde F$ of $H_t$ we have
\begin{equation}\label{4.32}
\begin{array}{ll}
\widetilde F_{(x,\omega )}(X^C+\alpha ^V,Y^C+\beta ^V,Z^C+\gamma ^V) \\
=\left(\widetilde F^\prime +\widetilde F^{''}\right)_{(x,\omega )}(X^C+\alpha ^V,Y^C+\beta ^V,Z^C+\gamma ^V) ,
\end{array}
\end{equation}
where $\widetilde F^\prime $ and $\widetilde F^{''}$ are determined by \eqref{4.31} and \eqref{4.26}, respectively.
\par
(i) Let us assume that $M$ is flat or ${\rm dim}M=2$. Then $\widetilde F^\prime =0$ and from \eqref{4.32} we obtain that $H_t \in \mathbb{G}_5$. Hence, according to the assertion (iv) from \thmref{Theorem A}, $H_t$ is quasi-para-Sasakian. Applying the assertion (ii) from \thmref{Theorem A} we conclude that $H_t$ is para-Sasakian if and only if $\mathbb{G}_5=\overline {\mathbb{G}}_5$. Analogously as in \thmref{Theorem 4.3} we establish that it is equivalent to $b=4a^2$.
\par
(ii) In the case when $M$ is not flat and ${\rm dim}M>2$ the equality \eqref{4.32} holds which means that $H_t \in \mathbb{G}_4\oplus \mathbb{G}_5$. By using the assertion (iii) from \thmref{Theorem A} we complete the proof.
\end{proof}

\end{document}